\documentclass[12pt,a4paper]{amsart}
\usepackage{amsfonts}
\usepackage{amsthm}
\usepackage{amsmath}
\usepackage{amscd}
\usepackage[latin2]{inputenc}
\usepackage{t1enc}
\usepackage[mathscr]{eucal}
\usepackage{indentfirst}
\usepackage{graphicx}
\usepackage{graphics}
\usepackage{pict2e}
\usepackage{epic}
\usepackage{pgfplots}
\usepackage[section]{placeins}
\usetikzlibrary{positioning}
\numberwithin{equation}{section}
\usepackage[margin=1.5cm]{geometry}
\usepackage{epstopdf} 
\usepackage{hyperref}
\hypersetup{
    colorlinks=true,
    linkcolor=magenta,
    filecolor=magenta,      
    urlcolor=magenta,
    citecolor=magenta
}

\theoremstyle{plain}
\newtheorem{Th}{Theorem}[section]
\newtheorem{Lemma}[Th]{Lemma}

\newtheorem{Prop}[Th]{Proposition}

 \theoremstyle{definition}
\newtheorem{Def}[Th]{Definition}

\newtheorem{Rem}[Th]{Remark}
\newtheorem{?}[Th]{Problem}

\usepackage{fancyhdr}
\begin{document}

\title{Asymptotic Enumeration of Binary Contingency Tables and Comparison with Independence Heuristic}

\author[Da Wu]{Da Wu}

\address{University of Pennsylvania \\ Department of Mathematics \\ David Rittenhouse Lab \\ 209 South 33rd Street \\ Philadelphia, PA, 10104-6395} 

\email{dawu@math.upenn.edu}



\begin{abstract}
 For parameters $n,\delta,B,C$, we obtained a sharp asymptotic formula for the number of $(n+\lfloor n^\delta\rfloor)^2$-dimensional binary contingency tables with non-uniform margins taking values of $\lfloor BCn\rfloor$ and $\lfloor Cn\rfloor$. Furthermore, we compared our sharp asymptotics with the classical \textit{independence heuristic estimate} and proved that the independence heuristic overestimates by a factor of $e^{\Theta(n^{2\delta})}$. Our comparison is based on the analysis of the \textit{correlation ratio} and an explicit bound for the constant in $\Theta$ is also obtained. 
\end{abstract}
\maketitle

\section{Introduction}
\subsection{Overview}
This paper studies the asymptotic enumeration of \textit{binary contingency tables} and its connection with the classical \textit{independence heuristic} introduced by I. J. Good $70$ years ago \cite{Good1}.
\par
Binary contingency table is the set of $0$-$1$ matrices with fixed row and column sums. Let $\mathbf r=(r_1,\ldots,r_m)$ and $\mathbf c=(c_1,\ldots,c_n)$ be two positive integer vectors with same total sum of entries, i.e. $\sum_{i=1}^m r_i=\sum_{j=1}^n c_j=N$. Let 
\begin{equation*}
	\mathcal M(\mathbf r,\mathbf c)=\left\lbrace X=(X_{ij}): \sum_{k=1}^{n} X_{ik}=r_i,\sum_{k=1}^m X_{kj}=c_j \ \text{for all}\ 1\leq i\leq m, 1\leq j\leq n  \right\rbrace
\end{equation*}
be the set of binary contingency tables with row margin $\mathbf r$ and column margin $\mathbf c$. Since $X_{ij}\in \lbrace 0,1\rbrace$, it is easy to see that $r_i\leq n, c_j\leq m$ for all $i,j$. Binary contingency table has close connections with bipartite graphs with fixed degree sequence, see e.g., \cite{ICM18}, for historical review. It also arises as the structural constants in the ring of symmetric functions and representation theory of general linear groups, see \cite{Mac98}. 
\par
Estimating the cardinality of $\mathcal M(\mathbf r,\mathbf c)$ is a fundamental problem in analytic combinatorics; see for instance, \cite{Good1} and \cite{Bar10}. At the very beginning, we have the following effortless estimate based on the so-called \textit{independence heuristic}. Precisely speaking, fix $\mathbf r=(r_1,\ldots, r_m)$ and $\mathbf c=(c_1,\ldots, c_n)$ with $N=r_1+\ldots+r_m=c_1+\ldots+c_n$, and let $\mathcal M_N$ be the set of $m\times n$ $0$-$1$ matrices with total sum of entries $N$. Let $X$ be the uniform sample from $\mathcal M_N$ and consider the following two events:
\begin{equation*}
	\mathcal R_{\mathbf r}=\lbrace \text{$X$ has row sum $\mathbf r$}\rbrace\qquad\text{and}\qquad \mathcal R_{\mathbf c}=\lbrace \text{$X$ has column sum $\mathbf c$}\rbrace.
\end{equation*}
It follows from the definition that 
\begin{equation*}
	\mathbb P\left(\mathcal R_{\mathbf r}\right)=\frac{|\mathcal R_{\mathbf r}|}{|\mathcal M_N|}\qquad\text{and}\qquad \mathbb P\left(\mathcal R_{\mathbf c}\right)=\frac{|\mathcal R_{\mathbf c}|}{|\mathcal M_N|}.
\end{equation*}
Assume $\mathcal R_{\mathbf r}$ and $\mathcal R_{\mathbf c}$ are independent, then 
\begin{equation*}
	\frac{|\mathcal R_{\mathbf r}|}{|\mathcal M_N|}\cdot\frac{|\mathcal R_{\mathbf c}|}{|\mathcal M_N|} =\mathbb P\left(\mathcal R_{\mathbf r}\right)\mathbb P\left(\mathcal R_{\mathbf c}\right)=\mathbb P(\mathcal R_{\mathbf r}\cap\mathcal R_{\mathbf c})=\frac{|\mathcal M(\mathbf r,\mathbf c)|}{|\mathcal M_N|}.
\end{equation*}
Therefore, 
\begin{align}\label{derivation of IH}
	|\mathcal M(\mathbf r,\mathbf c)|=\frac{|\mathcal R_{\mathbf r}|\cdot|\mathcal R_{\mathbf c}|}{|\mathcal M_N|}=\binom{mn}{N}^{-1}\prod_{i=1}^m \binom{n}{r_i}\prod_{j=1}^n \binom{m}{c_j}.
\end{align}
We call 
\begin{equation}\label{eqn:ihe}
	\mathcal I(\mathbf r,\mathbf c)=\binom{mn}{N}^{-1}\prod_{i=1}^m \binom{n}{r_i}\prod_{j=1}^n \binom{m}{c_j}.
\end{equation}
the \textit{independence heuristic estimate} corresponding to margins $\mathbf r$ and $\mathbf c$. 
\par
In $1969$, O'Neil \cite{Neil69} solved the problem for the case of large sparse binary matrices with $r_i,c_j\leq (\log n)^{\frac{1}{4}-\varepsilon}$ for all $1\leq i,j\leq n$. In $2010$, A. Barvinok \cite{Bar10} used \emph{permanents} and the \emph{van der Waerden bound} for doubly-stochastic matrix to obtain an estimate for generic $\mathbf r$ and $\mathbf c$. This estimate has recently been improved in \cite{Lor23} using the techniques of Lorentzian polynomials. Later on, Barvinok and Hartigan \cite{Bar11} used the \textit{Maximal Entropy Principle} and local central limit theorem to obtain a more precise asymptotic formula under certain regularity conditions. 
\par
Our main focus is to first derive the asymptotics of $|\mathcal M(\mathbf r,\mathbf c)|$ when both $\mathbf r$ and $\mathbf c$ take two different linear values of $n$; see the precise definition in the next subsection. After that, we compared the results with the independence heuristic estimate (\ref{eqn:ihe}) and showed that the independence heuristic estimate leads to a large overestimate by a factor of $e^{\Theta(n^{2\delta})}$.  Our derivation of asymptotics follows closely the spirit of \cite{PL22} and is mainly based on Barvinok's asymptotic formula in \cite{Bar10} and the author's recent work \cite{Wu22} on the limiting distribution of \emph{Random Binary Contingency Tables}. Similar asymptotics for the uniform margin case can also be obtaned with the same techniques and the limiting distribution derived in \cite{Wu23}. 
\subsection{Setup and Statements of Main Results}
For $0< \delta<1$, $0<B\leq\frac{1}{C}$ and $0<C<1$, let  
\begin{equation*}
	\widetilde{\mathbf r}=\widetilde{\mathbf c}:=(\underbrace{\lfloor BCn\rfloor ,\ldots, \lfloor BCn \rfloor}_{\text{$\lfloor n^\delta\rfloor$ entries}},\underbrace{\lfloor Cn\rfloor\ldots, \lfloor Cn\rfloor}_{\text{$n$ entries}})\in \mathbb N^{[n^\delta]+n},
\end{equation*}
and let 
\begin{equation*}
	\mathcal M_{n,\delta}(B,C):=\mathcal M(\widetilde{\mathbf r},\widetilde{\mathbf c}).
\end{equation*}
Namely, $\mathcal M_{n,\delta}(B,C)$ is the set of $(\lfloor n^\delta\rfloor+n)^2$-dimensional binary matrices whose first $\lfloor n^\delta\rfloor$ rows and columns have sum $\lfloor BCn\rfloor$ and remaining $n$ rows and columns have sum $\lfloor Cn\rfloor$. 
\par
The first main result of this paper is the sharp asymptotics of $|\mathcal M_{n,\delta}(B,C)|$.  
\begin{Th}\label{Main Theorem on Asymptotic Enumeration}
	Fix $0<\delta<1$, $n\geq 1$, $0<C<\frac{3}{4}$ and $0<B\leq \frac{1}{C}$,
	and let $f(x):=x\log \frac{1}{x}+(1-x)\log\frac{1}{1-x}$. We have 
	\begin{align*}
		\log |\mathcal M_{n,\delta}(B,C)| &=f(C)n^2+\left[2f(BC)-(BC)\log \left(\frac{1-C}{C}\right)\right]n^{1+\delta}\\
		&+\left[f(z_{11}^*)+z_{11}^*\log\left(\frac{1-C}{C}\cdot \frac{(BC)^2}{(1-BC)^2}\right)-\frac{B^2 C}{2(1-C)} \right]n^{2\delta}\\
		&+O(n^{3\delta-1}+n\log n),
	\end{align*}
	where $z_{11}^*=\frac{B^2(1-C)}{B^2-2B+1/C}$.
\end{Th}
This theorem will be proved in the Section \ref{sec:proof of main}. For now, we remark that the proof is based on the \textit{Maximum Entropy Principle} and the function $f(x)$ is Shannon-Boltzmann entropy of Bernoulli random variable with mean $x$; the coefficients in front of $n^2,n^{1+\delta},n^{2\delta}$ all come from the Bernoulli entropy. 
\par
Next, we study the relationship between $|\mathcal M_{n,\delta}(B,C)|$ and its corresponding independence heuristic. Recall that for margins $\mathbf r=(r_1,\ldots, r_m)$ and $\mathbf c=(c_1,\ldots,c_n)$, the \textit{independence heuristic estimate} is the following quantity, 
\begin{equation*}
	\mathcal I(\mathbf r,\mathbf c)=\binom{mn}{N}^{-1}\prod_{i=1}^m \binom{n}{r_i}\prod_{j=1}^n \binom{m}{c_j},
\end{equation*}    
where $N=\sum_{i=1}^m r_i=\sum_{j=1}^n c_j$ is the total sum of entries. We denote 
\begin{equation*}
	\mathcal I_{n,\delta}(B,C):=\mathcal I(\widetilde{\mathbf r}, \widetilde{\mathbf c}).
\end{equation*}
To study the relation between $\mathcal I_{n,\delta}(B,C)$ and $|\mathcal M_{n,\delta}(B,C)|$, we consider their \textit{Correlation Ratio} $\rho_{n,\delta}(B,C)$. It is defined as   
\begin{equation*}
	\rho_{n,\delta}(B,C):=\frac{|\mathcal M_{n,\delta}(B,C)|}{\mathcal I_{n,\delta}(B,C)}.
\end{equation*}
Our next main result is on the asymptotic behaviour of $\rho_{n,\delta}(B,C)$. 
\begin{Th}\label{Main Theorem on Correlation Ratio}
	Fix $0<\delta<1$, $n\geq 1$, $0<C<\frac{3}{4}$ and $0<B\leq\frac{1}{C}$, we have 
\begin{equation*}
	\lim_{n\to \infty}\frac{1}{n^2}\log \rho_{n,\delta}(B,C)=0 \qquad{and}\qquad \lim_{n\to \infty}\frac{1}{n^{1+\delta}}\log \rho_{n,\delta}(B,C)=0.
\end{equation*}
Furthermore, let 
\begin{equation*}
	\Delta_{B,C}:=\lim_{n\to \infty}\frac{1}{n^{2\delta}}\log \rho_{n,\delta}(B,C),
\end{equation*}
and we have the following explicit formula for $\Delta_{B,C}$: 
\begin{equation*}
	\Delta_{B,C}=1-\frac{B^2C-2BC+1}{1-C}-\log\left(\frac{1-C}{B^2C-2BC+1} \right).
\end{equation*}
Moreover,
\begin{equation*}
    0\geq \Delta_{B,C}>
    \begin{cases}
        -\frac{1}{C}+\log\frac{1}{C}+1 &\quad\text{when $0<C\leq \frac{1}{2}$},\\
         -\frac{1}{1-C}+\log\frac{1}{1-C}+1 &\quad\text{when $\frac{1}{2}<C<\frac{3}{4}$},
    \end{cases}
\end{equation*}
with $\Delta_{B,C}=0$ if and only if $B=1$. 
\end{Th}
The behaviour of $\Delta_{B,C}$ tells us that the independence heuristic overestimates the number of tables in $\mathcal M_{n,\delta}(B,C)$. This matches Barvinok's arguments on \textit{cloned margins}, see \cite{Bar10} for details. Probabilistically speaking, the events
\begin{equation*}
	\mathcal R_{n,\delta}(B,C) =\lbrace \text{$0$-$1$ matrices has row sums $\widetilde{\mathbf r}$}\rbrace
\end{equation*}
and 
\begin{equation*}
	C_{n,\delta}(B,C) =\lbrace \text{$0$-$1$ matrices has column sums $\widetilde{\mathbf c}$}\rbrace
\end{equation*}
are asymptotically \textit{negatively correlated} instead of being asymptotically independent. $\Delta_{B,C}$ quantifies how far they are away from being independent. In fact, when $B=1$, i.e. all row sums and columns sums are equal, the independence heuristic provides the best estimate. As $B$ moves away from $1$, meaning that the margins become less and less uniform, the independence heuristic overestimates by a factor of $e^{\Theta(n^{2\delta})}$.
\par
It is really interesting for the readers to compare our results with the recent work of \cite{PL22} on non-negative integer case. When the contingency tables are non-negative integer valued, the independence heuristic leads to a large undercounting, which is opposite to our binary case. The reason behind this phenomena remain mysterious. It would be nice if we can obtain the intermediate results on the contingency tables whose entries take values from $\lbrace 0,1,\ldots, k\rbrace$ for finite $k$. 
\begin{figure}[hbt!]
\centering
	\begin{tikzpicture}
\begin{axis}[domain=0:3.5, samples=100,
    restrict y to domain=-4:4,xlabel=$B$,ylabel=$\Delta_{B,C}$ for different values of $C$, legend pos=outer north east]
\addplot [color=red, dotted] [domain=0:2]   {-x*x+2*x-1+ln(x*x-2*x+2)};
\addplot [color=blue, densely dotted] [domain=0:4]   {1-(1/3)*(x*x-2*x+4)-ln(3)+ln(x*x-2*x+4)};
\addplot [color=black, loosely dotted] [domain=0:1.6]   {1-(1/3)*(5*x*x-10*x+8)-ln(3)+ln(5*x*x-10*x+8)};
\addplot [color=magenta] [domain=0:4]   {1-(1/7)*(x*x-2*x+8)-ln(7)+ln(x*x-2*x+8)};
\legend{$C=\frac{1}{2}$, $C=\frac{1}{4}$, $C=\frac{5}{8}$, $C=\frac{1}{8}$}
\end{axis}
\end{tikzpicture}
\caption{Plot of $\Delta_{B,C}=\lim_{n\to \infty}\frac{1}{n^{2\delta}}\log\rho_{n,\delta}(B,C)$ as a function of $B\in [0,4]$ for $4$ different fixed values of $C$. For each $0<C<\frac{3}{4}$, $B$ is allowed to take value from $0$ to $\frac{1}{C}$.} 
\end{figure} 
\section{Proof of Theorem \ref{Main Theorem on Asymptotic Enumeration}}\label{sec:proof of main}
Let $\mathbf r=(r_1,\ldots,r_m)\in \mathbb N^m$ and $\mathbf c=(c_1,\ldots,c_n)\in \mathbb N^n$ be two positive integer vectors with total sum of entries. The \textit{binary transportation polytope} $\mathcal P(\mathbf r,\mathbf c)$ is defined as  
\begin{equation*}
	\mathcal P(\mathbf r,\mathbf c):=\left\lbrace X=(X_{ij})\in [0,1]^{mn}: \sum_{k=1}^n X_{ik}=r_i, \sum_{k=1}^m X_{kj}=c_j,\forall i,j \right\rbrace.
\end{equation*}
Barvinok introduced the following notion of \textit{Typical Table} in \cite{Bar10}. 
\begin{Def}[Typical Table]
	Let $\mathbf r=(r_1,\ldots,r_m)$, $\mathbf c=(c_1,\ldots,c_n)$ be two positive integer vectors with total sum of entries. For each $X=(X_{ij})\in \mathcal P(\mathbf r,\mathbf c)$, let 
	\begin{equation*}
		g(X)=\sum_{i,j}f(X_{ij}),
	\end{equation*} 
	where $f(x)=x\log\frac{1}{x}+(1-x)\log\frac{1}{1-x}$ for $x\in (0,1)$. The \textit{Typical Table} $Z=(z_{ij})$ is defined as the unique maximizer of $g$ on $\mathcal P(\mathbf r,\mathbf c)$.  
\end{Def}
\begin{Rem}
	The function $g$ is strictly concave so it attains a unique maximum in the interior of the binary transportation polytope. Hence typical table is well-defined. 
\end{Rem}
\begin{Rem}
	The function $f(x)=x\log\frac{1}{x}+(1-x)\log\frac{1}{1-x}$ is the Shannon-Boltzmann entropy of Bernoulli random variable with mean $x$. Therefore, $g(X)$ can be viewed as the Bernoulli entropy of $X$ and $Z$ is the \textit{maximal entropy matrix} on the polytope $\mathcal P(\mathbf r,\mathbf c)$. 
\end{Rem}
\begin{Rem}
	By symmetry, two entries in typical table are equal if they have the same margin conditions, i.e $z_{ij}=z_{i'j'}$ if $r_i=r_{i'}, c_j=c_{j'}$. 
\end{Rem}
It turns out that typical table has close connection to the cardinality of $\mathcal M(\mathbf r,\mathbf c)$. The following theorem proved by Barvinok in \cite{Bar10} plays a key role in our proof. 
\begin{Th}\label{Barvinok's theorem}
	Fix row margins $\mathbf r=(r_1,\ldots,r_m)$ and column margins $\mathbf c=(c_1,\ldots,c_n)$ and let $Z=(z_{ij})$ be the typical table associated with $\mathcal P(\mathbf r,\mathbf c)$. There exists some absolute constant $\gamma>0$ such that 
	\begin{equation}
		(mn)^{-\gamma(m+n)} e^{g(Z)}\leq |\mathcal M(\mathbf r,\mathbf c)|\leq e^{g(Z)}.
	\end{equation} 
\end{Th}
Next, by \cite[(2.3)]{Wu22} and \cite[Lemma $2.4$]{Wu22}, we have the following asymptotics of entries of typical table $Z=(z_{ij})$.
\begin{Lemma}
	Fix $0\leq \delta<1$, $0<C<\frac{3}{4}$ and $0<B\leq \frac{1}{C}$. Let $Z=(z_{ij})$ be the typical table for $\mathcal M_{n,\delta}(B,C)$. Then there exists constants $\gamma_1(B,C)$ and $\gamma_2(B,C)$ such that the followings hold:
	\begin{enumerate}
		\item $\left|z_{11}-\frac{B^2(1-C)}{B^2-2B+1/C}\right| \leq \gamma_1(B,C) n^{\delta-1}$,
		\item $\left|z_{1,n+1}-BC\right| \leq \gamma_2(B,C) n^{\delta-1}$,
		\item $\left|z_{n+1,n+1}-C\right| \leq BC n^{\delta-1}$.
	\end{enumerate}
\end{Lemma}
Next, we derive the asymptotics of $g(Z)$, which is the entropy for the typical table. Our proof is almost identical to the \cite[Proposition $3.4$]{PL22}, except that we plug in different limits of the typical table. Notice that in the binary case, there is no sharp phase transition for the value of $B$ with respect to $C$; see \cite[Remark 1.4]{Wu22} for more discussions on this aspect.
\begin{Prop}\label{asymptotics of typical table}
	Fix $0<\delta<1$, $0<C<\frac{3}{4}$ and $0<B\leq \frac{1}{C}$. Let $Z=(z_{ij})$ be the typical table for $\mathcal M_{n,\delta}(B,C)$. Then 
	\begin{align*}
		g(Z) &=n^2 f(z_{n+1,n+1})+2n[n^\delta] f(z_{1,n+1})+[n^\delta]^2 f(z_{11})\\
	&=f(C)n^2+\left[2f(BC)-(BC)\log \left(\frac{1-C}{C}\right)\right]n^{1+\delta}\\
		&+\left[f(z_{11}^*)+z_{11}^*\log\left(\frac{1-C}{C}\cdot \frac{(BC)^2}{(1-BC)^2}\right)-\frac{B^2 C}{2(1-C)} \right]n^{2\delta}\\
		&+O(n^{3\delta-1})+O(n),
	\end{align*}
	where $z_{11}^*=\frac{B^2(1-C)}{B^2-2B+1/C}$. 
\end{Prop}
\begin{proof}
By symmetry of typical table and marginal condition,  
\begin{equation*}
	\begin{cases}
		[n^\delta]z_{11}+n z_{1,n+1}=[BCn]\\
		[n^\delta]z_{1,n+1}+n z_{n+1,n+1}=[Cn]
	\end{cases}.
\end{equation*}
Therefore, 
\begin{equation*}
	\begin{cases}
		([n^\delta]/n)z_{11}+z_{1,n+1}=[BCn]/n=BC+O(n^{-1})\\
		([n^\delta]/n)z_{1,n+1}+z_{n+1,n+1}=[Cn]/n=C+O(n^{-1})
	\end{cases}
\end{equation*}
Let $z_{11}^*=\frac{B^2(1-C)}{B^2-2B+1/C}$, then
\begin{align*}
	C-z_{n+1,n+1} &=z_{1,n+1}([n^\delta]/n)+O(n^{-1})\\
	&=(BC)([n^\delta]/n)-(BC-z_{1,n+1})([n^\delta]/n)+O(n^{-1})\\
	&=(BC)([n^\delta]/n)-z_{11}([n^\delta]/n)^2+O(n^{-1})\\
	&=(BC)([n^\delta]/n)-z_{11}^*([n^\delta]/n)^2-(z_{11}-z_{11}^*)([n^\delta]/n)^2+O(n^{-1})\\
	&=(BC)([n^\delta]/n)-z_{11}^*([n^\delta]/n)^2+O(n^{3\delta-3}+n^{-1}).
\end{align*}
Similarly, 
\begin{align*}
	BC-z_{1,n+1} &=([n^\delta]/n)z_{11}+O(n^{-1})\\
	&=z_{11}^* ([n^\delta]/n)+(z_{11}-z_{11}^*)([n^\delta]/n)+O(n^{-1})\\
	&=z_{11}^* ([n^\delta]/n)+O(n^{2\delta-2}+n^{-1}).
\end{align*}
Taylor expansion of $f(x)$ around $a$ has the following form, 
\begin{equation*}
	f(x)=f(a)+\log \left(\frac{1-a}{a}\right)(x-a)+\frac{1}{2(a-1)a}(x-a)^2+O\left(|x-a|^3\right).
\end{equation*}
Hence, 
\begin{align*}
	f(z_{n+1,n+1}) &= f(C)-\log\left(\frac{1-C}{C}\right)(BC) ([n^\delta]/n)+\log\left(\frac{1-C}{C}\right)z_{11}^*([n^\delta]/n)^2\\
	&+\frac{B^2 C}{2(C-1)}([n^\delta]/n)^2+O\left(n^{3\delta-3}\right)+O\left(n^{\delta-1}\right)\\
	f(z_{1,n+1}) &=f(BC)-\log\left(\frac{1-BC}{BC}\right)z_{11}^*([n^\delta]/n)+\frac{1}{2(BC-1)BC}\left(z_{11}^*\right)^2([n^\delta]/n)^2\\
	&+O\left(n^{2\delta-2}\right)+O\left(n^{-1}\right)\\
	f(z_{11}) &= f\left(z_{11}^*\right)+O\left(n^{\delta-1}\right)+O\left(n^{-1}\right).
\end{align*}
Therefore, we have that 
\begin{align*}
	g(Z) &=n^2 f(z_{n+1,n+1})+2n[n^\delta] f(z_{1,n+1})+[n^\delta]^2 f(z_{11})\\
	&=f(C)n^2+\left[2f(BC)-(BC)\log \left(\frac{1-C}{C}\right)\right]n^{1+\delta}\\
		&+\left[f(z_{11}^*)+z_{11}^*\log\left(\frac{1-C}{C}\cdot \frac{(BC)^2}{(1-BC)^2}\right)-\frac{B^2 C}{2(1-C)} \right]n^{2\delta}\\
		&+O(n^{3\delta-1})+O(n).
\end{align*}
\end{proof}
\begin{proof}[Proof of Theorem \ref{Main Theorem on Asymptotic Enumeration}]
	Notice that for margins $\widetilde{\mathbf r}$ and $\widetilde{\mathbf c}$, by Theorem \ref{Barvinok's theorem},
\begin{equation}
	g(Z)-\gamma'n\log n\leq \log|\mathcal M_{n,\delta}(B,C)|\leq g(Z)	
\end{equation}
Theorem $\ref{Main Theorem on Asymptotic Enumeration}$ follows from Proposition \ref{asymptotics of typical table}.
\end{proof}
\section{Independence Heuristic Estimate}
Using Stirling formula, we can deduce the asymptotics of \textit{independence heuristic estimation}. Similar asymptotics of independence heuristic for non-negative integer-valued contingency tables can be found in \cite[Lemma 4.1]{PL22}.  
\begin{Lemma}\label{Lemma on asymptotics of independent heuristic}
	Fix $0<\delta<1$, $0<C<1$ and $0<B\leq\frac{1}{C}$, 
	\begin{align*}
	\log\mathcal I_{n,\delta}(B,C) &=f(C)n^2+\left[2f(BC)-BC\log\left(\frac{1-C}{C}\right)\right]n^{1+\delta}\\
	&+\left[\frac{B^2C-4BC+2C}{2(1-C)}+\log(1-C)-2\log(1-BC)\right]n^{2\delta}\\
	&+O\left(n^{3\delta-1}+n\log n\right).
\end{align*}
\end{Lemma}
\begin{proof}
By Stirling formula, we have that
\begin{equation*}
	\log \binom {a+b}{a}=(a+b)\log(a+b)-a\log a-b\log b+O\left(\log(a+b)\right).
\end{equation*}
Recall that for general margins $\mathbf r$ and $\mathbf c$, the independence heuristic estimate takes the form 
\begin{equation*}
	\mathcal I(\mathbf r,\mathbf c)=\binom{mn}{N}^{-1}\prod_{i=1}^m \binom{n}{r_i}\prod_{j=1}^n\binom{m}{c_j}.
\end{equation*}
Therefore, 
\begin{align*}
	\log\mathcal I(\mathbf r,\mathbf c) &=N\log N+(mn-N)\log(mn-N)-\sum_{i=1}^m r_i\log r_i-\sum_{i=1}^m (n-r_i)\log(n-r_i)\\
	&-\sum_{j=1}^n c_j\log c_j-\sum_{j=1}^n (m-c_j)\log(m-c_j)+O\left((m+n)\log (mn) \right).
\end{align*}
In our setup, when $\mathbf r=\widetilde{\mathbf r}$ and $\mathbf c=\widetilde{\mathbf c}$, the dimension is $n+n^\delta$ and the total sum of entries $N=BCn^{1+\delta}+Cn^2$. By Taylor expansion 
\begin{equation*}
	\log(x+y)=\log y+\frac{x}{y}-\frac{x^2}{2y^2}+O\left(\frac{x^3}{y^3}\right),
\end{equation*}
we have that 
\begin{align*}
	\log(N) &=\log(Cn^2+BCn^{1+\delta})=\log(Cn^2)+Bn^{\delta-1}-\frac{B^2}{2}n^{2\delta-2}+O\left(n^{3\delta-3}\right)\\
	\log(mn-N) &=\log\left((1-C)n^2+(2-BC)n^{1+\delta}+n^{2\delta} \right)=\log\left[(1-C)n^2\right]\\
	&+\left(\frac{2-BC}{1-C}\right)n^{\delta-1}+\frac{-2-2C+4BC-B^2C^2}{2(1-C)^2}n^{2\delta-2}+O\left(n^{3\delta-3}\right)\\
	\log\left[(1-C)n+n^\delta\right] &=\log[(1-C)n]+\frac{n^{\delta-1}}{1-C}-\frac{n^{2\delta-2}}{2(1-C)^2}+O\left(n^{3\delta-3}\right)\\
	\log\left[(1-BC)n+n^\delta\right] &=\log \left[ (1-BC)n\right]+\frac{n^{\delta-1}}{1-BC}-\frac{n^{2\delta-2}}{2(1-BC)^2}+O\left(n^{3\delta-3}\right)
\end{align*}
Therefore, $\log\mathcal I_{n,\delta}(B,C)$ has the following expansion, 
\begin{align*}
	&\log\mathcal I_{n,\delta}(B,C) \\
	&=\left(Cn^2+BCn^{1+\delta}\right)\left[\log(Cn^2)+Bn^{\delta-1}-\frac{B^2}{2}n^{2\delta-2}\right]\\
	&+\left[(1-C)n^2+(2-BC)n^{1+\delta}+n^{2\delta}\right]\times \\
	&\left[\log((1-C)n^2)+\left(\frac{2-BC}{1-C}\right)n^{\delta-1}+\frac{-2-2C+4BC-B^2C^2}{2(1-C)^2}n^{2\delta-2} \right]\\
	&-2BC n^{1+\delta}\log (BCn)-2n^\delta\left[(1-BC)n+n^\delta\right]\left[\log((1-BC)n)+\frac{n^{\delta-1}}{1-BC}-\frac{n^{2\delta-2}}{2(1-BC)^2}\right]\\
	&-2Cn^2\log (Cn)-2n\left[(1-C)n+n^\delta\right]\cdot\left[\log((1-C)n)+\frac{n^{\delta-1}}{1-C}-\frac{n^{2\delta-2}}{2(1-C)^2} \right]\\&+O\left(n^{3\delta-1}+n\log n \right).
\end{align*}
After reorganizing terms, we have that
\begin{align*}
	\log\mathcal I_{n,\delta}(B,C) &=f(C)n^2+\left[2f(BC)-BC\log\left(\frac{1-C}{C}\right)\right]n^{1+\delta}\\
	&+\left[\frac{B^2C-4BC+2C}{2(1-C)}+\log(1-C)-2\log(1-BC)\right]n^{2\delta}\\
	&+O\left(n^{3\delta-1}+n\log n\right).
\end{align*}
\end{proof}
\begin{proof}[Proof of Theorem \ref{Main Theorem on Correlation Ratio}]
	By Theorem $\ref{Main Theorem on Asymptotic Enumeration}$, 
	\begin{align*}
		\log \left|\mathcal M_{n,\delta}(B,C)\right| &=f(C)n^2+\left[2f(BC)-(BC)\log \left(\frac{1-C}{C}\right)\right]n^{1+\delta}\\
		&+\left[f(z_{11}^*)+z_{11}^*\log\left(\frac{1-C}{C}\cdot \frac{(BC)^2}{(1-BC)^2}\right)-\frac{B^2 C}{2(1-C)} \right]n^{2\delta}\\
		&+O(n^{3\delta-1}+n\log n).
	\end{align*}
	On the other hand, by Lemma \ref{Lemma on asymptotics of independent heuristic},  
	\begin{align*}
	\log\mathcal I_{n,\delta}(B,C) &=f(C)n^2+\left[2f(BC)-BC\log\left(\frac{1-C}{C}\right)\right]n^{1+\delta}\\
	&+\left[\frac{B^2C-4BC+2C}{2(1-C)}+\log(1-C)-2\log(1-BC)\right]n^{2\delta}\\
	&+O\left(n^{3\delta-1}+n\log n\right).
\end{align*}
From this, we can easily deduce that 
\begin{align*}
	\lim_{n\to \infty}\frac{1}{n^{2}}\log \frac{|\mathcal M_{n,\delta}(B,C)|}{\mathcal I_{n,\delta}(B,C)}=0, \qquad \lim_{n\to \infty}\frac{1}{n^{1+\delta}}\log \frac{|\mathcal M_{n,\delta}(B,C)|}{\mathcal I_{n,\delta}(B,C)}=0
\end{align*}
and 
\begin{align*}
	\lim_{n\to \infty}\frac{1}{n^{2\delta}}\log \frac{|\mathcal M_{n,\delta}(B,C)|}{\mathcal I_{n,\delta}(B,C)}=\Delta_{B,C},
\end{align*}
where
\begin{align*}
	&-\Delta_{B,C}\\
	&=\frac{B^2C-4BC+2C}{2(1-C)}+\log\left(\frac{1-C}{(1-BC)^2}\right)-f(z_{11}^*)\\
 &-z_{11}^*\log\left(\frac{1-C}{(1-BC)^2}\right)-z^*_{11}\log(B^2C)+\frac{B^2C}{2(1-C)}.
\end{align*}
Notice that 
	\begin{equation*}
		1-z_{11}^*=\frac{(BC-1)^2}{B^2C-2BC+1},
	\end{equation*}
	By direct computation, 
	\begin{align*}
		-\Delta_{B,C} &=\frac{2B^2C-4BC+2C}{2(1-C)}+\frac{(BC-1)^2}{B^2C-2BC+1}\log\left[\frac{1-C}{B^2C-2BC+1}\right]\\
		&+\frac{B^2(1-C)}{B^2-2B+1/C}\log\left[\frac{1-C}{B^2C-2BC+1} \right]\\
		&=\frac{B^2C-2BC+1}{1-C}+\log\left[\frac{1-C}{B^2C-2BC+1}\right]-1\\
		&\geq 0,
	\end{align*}
	where the last step is based on the simple fact that $x-\log x\geq 1$ for all $x>0$. For fixed $C$, the explicit bound of $\Delta_{B,C}$ follows from the elementary analysis. 
\end{proof} 
\textit{Acknowledgement}: I would like to thank Robin Pemantle for many helpful discussions. 

\bibliographystyle{alpha}
\bibliography{bibliography}

\begin{thebibliography}{BLP23}

\bibitem[Bar10]{Bar10}
Alexander Barvinok.
\newblock On the number of matrices and a random matrix with prescribed row and
  column sums and 0-1 entries.
\newblock {\em Advances in Mathematics}, 224(1):316--339, 2010.

\bibitem[BH13]{Bar11}
Alexander Barvinok and J.~Hartigan.
\newblock The number of graphs and a random graph with a given degree sequence.
\newblock {\em Random Structures and Algorithms}, 42, 05 2013.

\bibitem[BLP23]{Lor23}
Petter Br\"{a}nd\'{e}n, Jonathan Leake, and Igor Pak.
\newblock Lower bounds for contingency tables via lorentzian polynomials.
\newblock {\em Israel Journal of Mathematics}, 253:43--90, 03 2023.

\bibitem[GC77]{Good1}
I.J. Good and J.F. Crook.
\newblock The enumeration of arrays and a generalization related to contingency
  tables.
\newblock {\em Discrete Mathematics}, 19(1):23--45, 1977.

\bibitem[LP22]{PL22}
Hanbaek Lyu and Igor Pak.
\newblock On the number of contingency tables and the independence heuristic.
\newblock {\em Bulletin of the London Mathematical Society}, 54(1):242--255,
  2022.

\bibitem[Mac98]{Mac98}
I.G. Macdonald.
\newblock {\em Symmetric Functions and Hall Polynomials}.
\newblock Oxford classic texts in the physical sciences. Clarendon Press, 1998.

\bibitem[O'N69]{Neil69}
Patrick~Eugene O'Neil.
\newblock {Asymptotics and random matrices with row-sum and column
  sum-restrictions}.
\newblock {\em Bulletin of the American Mathematical Society}, 75(6):1276 --
  1282, 1969.

\bibitem[Wor19]{ICM18}
Nicholas Wormald.
\newblock Asymptotic enumeration of graphs with given degree sequence.
\newblock {\em Proceedings of the International Congress of Mathematicians (ICM
  2018)}, pages 3245--3264, 2019.

\bibitem[Wu22]{Wu22}
Da~Wu.
\newblock On properties of random binary contingency tables with non-uniform
  margin.
\newblock {\em arXiv:2002.12559}, 2022.

\bibitem[Wu23]{Wu23}
Da~Wu.
\newblock Asymptotic properties of random contingency tables with uniform
  margin.
\newblock {\em Journal of Theoretical Probability}, 2023.

\end{thebibliography}

\end{document}